\documentclass[12pt, twoside]{amsart}
\usepackage{geometry}
\usepackage{tikz}
\usetikzlibrary{matrix}
\allowdisplaybreaks

\RequirePackage[OT1]{fontenc}
\RequirePackage{amsthm,amsmath,amsfonts}
\RequirePackage[numbers]{natbib}
\RequirePackage[colorlinks,citecolor=blue,urlcolor=blue]{hyperref}

\newtheorem{theorem}{Theorem}[section]
   \newtheorem{lemma}[theorem]{Lemma}

   \newtheorem{definition}[theorem]{Definition}

   \newtheorem{remark}[theorem]{Remark}
   
\numberwithin{equation}{section}

    \newcommand \id{1\!\!1}
    \newcommand\graph[1]{[\! [ #1]\! ]}
    
    \newcommand\cro[1]{\langle #1\rangle}
    
    \newcommand\norm[1]{|\!| #1|\!|}
    \newcommand\wh{\widehat}
    
    \newcommand\F{{\mathcal F}}
    \newcommand\G{{\mathcal G}}
    \newcommand\Gz{{\mathcal G}}
    \newcommand\ff{{\mathbb F}}
    
    \renewcommand\gg{{\mathbb G}}
    \newcommand\ggz{{\mathbb G}}
    \renewcommand\P{\mathbb P}
    \newcommand\Q{\mathbb Q}
    \newcommand\R{\mathbb R}
    \newcommand \E {{\mathbb E}}
    \newcommand\ya{{Y^a}}

    \newcommand{\is}{\cdot}
    \newcommand{\Rbrack}{[\![}
    \newcommand{\Lbrack}{]\!]}
    
    \newcommand\ooo[1]{\!\!\!\phantom{I}^o #1}

\newcommand\oo[1]{#1^o}

\begin{document}

\title[Entropy and additional utility of a discrete information]{Entropy and additional utility of a discrete information disclosed progressively in time}

\author[Anna Aksamit]{Anna Aksamit}

\address{Anna Aksamit\\
School of Mathematics and Statistics \\ University of Sydney
\\ Sydney, NSW 2006, Australia}
\email{anna.aksamit@sydney.edu.au}

\date{\today}

\maketitle

\begin{abstract}
The additional information carried by enlarged filtration and its measurement was studied by several authors. Already Meyer (Sur un theoreme de J. Jacod, 1978) and Yor (Entropie d'une partition, et grossissement initial d'une filtration, 1985), investigated stability of martingale spaces with respect to initial enlargement with atomic sigma-field. We extend these considerations to the case where information is disclosed progressively in time. We define the entropy of such information and we prove that its finiteness is enough for stability of some martingale spaces in progressive setting. Finally we calculate additional logarithmic utility of a discrete information disclosed progressively in time.
\end{abstract}

\section{Introduction}
The additional information carried by an enlarged filtration and its
measurement was studied by several authors. Already in
Meyer \cite{meyer1978theoreme} and Yor \cite{entropieyor}, the question on
stability of martingale spaces with respect to initial enlargement
with atomic $\sigma$-field was asked. 
Here we complete previous studies by giving a connection between progressive enlargement with thin random time, \cite{thin}, and \emph{conditional} entropy of a
partition associated to this time.
Our notion of the entropy of thin random time answers the question posted by Paul-Andr\'e Meyer in \cite{meyer1978theoreme} about additional knowledge associated with a partition and disclosed in progressive manner:

\emph{A similar problem, but perhaps of more interest, consists in measuring the resulting perturbation, in a probabilistic system, not by requiring knowledge at the instant 0, but by adding them progressively to the system.}\footnote{The original French version:\\ \emph{Un probl\`eme voisin, mais plus int\'eressant peut-\^etre, consiste \`a mesurer le bouleversement produit, sur un syst\`eme probabiliste, non pas en for\c cant des connaissances \`a l'instant $0$, mais en les for\c cant progressivement dans le syst\`eme.}}

There are also more recent related studies that generalise and extend earlier results to a variety of settings; see, for instance, \cite{AIS, ADIshannon, ADI:ef}. In contrast, our approach seeks to derive more explicit results and conditions by exploiting additional structure in the problem formulation.

In the case of initial enlargement with a partition $\mathcal C:=(C_n)_n$ -- that is the case where $\mathcal C$ is added to the reference filtration $\mathbb F$ at time zero forming filtration $\mathbb F^{\,\mathcal C}$ -- the additional knowledge is measured by entropy \cite{meyer_general, entropieyor}, namely
$$
H(\mathcal C):=-\sum_n\P(C_n)\log\P(C_n).
$$
In the case of progressive enlargement with a thin random time $\tau$, we
suggest the measurement of additional knowledge by the entropy of a thin random time defined through:
\begin{equation}
\label{entropytau}
H(\tau):=-\sum_n\E\left[\id_{\{\tau=T_n\}}\log \mathbb P(\tau=T_n<\infty|\mathcal F_{T_n})\right] ,
\end{equation}
where $(T_n)$ is an exhausting sequence for $\tau$, i.e., a family of $\mathbb F$-stopping times with disjoint graphs such that $\graph{\tau}\subset \bigcup_n\graph{T_n}$. 
Let us remark that the condition $H(\tau)<\infty$ is weaker than the condition $H(C)<\infty$.

Main results in Theorems \ref{decomposition2}, \ref{entropy2} and \ref{th:logu} 
justify that the notion in \eqref{entropytau} is the right one to measure perturbation in progressive case.

\section{General progressive enlargement with a partition}
\label{s:2}

\subsection{Preliminaries}

Let $(\Omega,\mathcal G,\mathbb P)$ be a complete
probability space, equipped with a filtration $\mathbb{F}:=(\F_t)_{t\geq 0}$ satisfying the usual conditions of completeness and right-continuity, and such that
$\F_\infty:=\bigcup_{t>0}\F_t \subset \mathcal G$.
For any
c\`adl\`ag process $X$ we   denote by $X_-$ the left-continuous
version of $X$, by $\Delta X$ the jump of $X$ and by $X_\infty$
the limit $\lim_{t\to \infty} X_t$ if it exists. The process $X$
is said to be increasing if, for almost all $\omega$, it satisfies
$X_t(\omega)\geq X_s(\omega)$ for all $t\geq s$. A random variable
is said to be positive if it has values in $[0, \infty)$. We
denote by $G\is X$ the stochastic integral of a predictable
process $G$ w.r.t. a semimartingale $X$, when this integral is well defined.

Consider a random time $\tau$, i.e., a $[0, \infty]$-valued
$\G$-measurable random variable. Note that a random time $\tau$ is
not necessarily $\F_\infty$-measurable. For a random time $\tau$,
we denote by $\graph{\tau}:=\{(\omega, t)\subset \Omega\times
\R^+: \tau(\omega)=t\}$ its graph. Let us recall, following
\cite{j}, some useful processes associated with  the pair
$(\ff,\tau)$. For the process
$A:=\id_{\Rbrack\tau,\infty\Rbrack}$, we denote by 
$A^o$ its $\ff$-dual
optional projection (for reader's convenience definition is recalled in  Appendix \ref{projections}). 
We also define two $\ff$-supermartingale
$Z$ as the optional projection of processes $1-A$, i.e.,
\begin{equation*}
Z_t:=\; ^o \Big [\id_{\Rbrack 0, \tau\Rbrack}\Big ]_t=\P(\tau>t|\F_t).
\end{equation*}
Then, the following process is a BMO $\ff$-martingale (see
\cite[Chapitre IV, section 1]{j}):
\begin{equation}
\label{martm}
m=A^o+Z.
\end{equation}

\subsection{Adding a partition}
In the subsection we generalise the progressive enlargement with a thin time studied in \cite{thin} and initial enlargement with a discrete random variable. We study the enlargement with a partition which is added to a reference filtration in an arbitrary way (allows for adding several members
of a partition at the same time).

More precisely, following the Definition 2.1 in \cite{thin}, let $\tau$ be a
thin time with exhausting sequence $(T_n)_{n\geq  0}$, i.e., the graphs satisfy $\tau$ satisfies $\graph{\tau}\subset \bigcup_n\graph{T_n}$ and $\graph{T_n}\cap \graph{T_m}$ for $n\neq m$, and 
let $\xi$ be a
discrete random variable with values in $\mathbb N$. We introduce the following family of sets
$$C_{n}^{k}=\{\tau=T_n<\infty, \xi=k\}\quad n,k\in \mathbb N\quad \textrm{and}\quad C_{0}^{0}=\{\tau=\infty\}$$ 
which form a partition of $\Omega$.
Note that $\{\tau=T_n<\infty\}=\{T_n<\infty\}\cap\,\bigcup_k C_n^k$ and $\sigma(\xi\id_{\{\tau=T_n\}})=\sigma((C^k_n)_k)$. 
Next, we define
\begin{equation}
\label{znk}
z^{n,k}_t:=\P(C_{n,k}|\F_t).
\end{equation}
We are interested by  the progressively enlarged filtration
$\ggz:=(\Gz_t)_{t\geq 0}$, which is the
smallest right-continuous filtration $\gg$ making $\tau$ a stopping time and $\xi$
a $\G_\tau$-measurable random variable, defined through
\begin{equation}
\label{xitau}
\Gz_t:=\bigcap_{s>t}\F_s\lor \sigma\left(\,\xi\id_{\{u\geq \tau\}}: \, u\leq s\right ).
\end{equation}

Similarly to \cite[Lemma 1.5]{thin} we obtain a key lemma for computing conditional expectations w.r.t. members of $\ggz$ in terms of conditional expectations of members of $\ff$, in the following lemma. 
\begin{lemma}
\label{key2}
For any $\G$-measurable integrable random variable $X$ and $s\leq t$ we have
$z^{n,k}_t>0$ and $z^{n,k}_{t-}>0$ for all $t\geq 0$ a.s. on $C_n^k$ for each $n, k$, and 
\[
\E\left[X|\G_t\right]\id_{\{s\geq T_n\}\cap C_n^k}
=\id_{\{s\geq T_n\}\cap C_n^k}\frac{\E\left[X \id_{ C_n^k}|\F_t\right]}{z^{n,k}_t}.
\]
\end{lemma}

In Theorem \ref{decomposition2} we include an important result about the behaviour of $\ff$-martingales as processes in $\ggz$, called $({\mathcal H}^\prime)$-hypothesis. 
Let us first recall that hypothesis $({\mathcal H}^\prime)$ holds for $(\ff, \ggz)$ if any $\ff$-martingale is an $\ggz$-semimartingale.  

\begin{theorem}
\label{decomposition2}
The hypothesis $({\mathcal H}^\prime)$ is satisfied for  $(\ff, \ggz)$.
Moreover, for each $\ggz$-predictable and bounded process $G$ and each $\ff$-local martingale $Y$ the integral $X:=\int G dY$ is an $\ggz$-semimartingale with the canonical decomposition
\begin{equation}
\label{thinH2} X_t= \widehat X_t + \int_0^{t\land
\tau}\frac{G_s}{Z_{s-}}d\cro{Y,m}_s +\sum_{n,k} \id_{C_{n,k}}
\int_0^t \id_{\{s>T_n\}}\frac{G_s}{z^{n,k}_{s-}}d\cro{Y,z^{n,k}}_s
\end{equation}
where $\widehat X$ is an $\ggz$-local martingale.
\end{theorem}
\begin{proof}
The proof is similar to the original proof of \cite[Theorem 4.1]{thin}. The first part follows from Jacod's result \cite{jacod1985grossissement} and Stricker's Theorem \cite[Theorem 4, Chapter II]{protter} since $\ggz \subset \ff\lor \sigma(C_n^k: k,n)$.
We then need to establish the form of a $\ggz$-predictable process $H$. 

\cite[Lemma (4,4)]{j} implies that
\[
H_t=\id_{\{t\leq \tau\}}J_t+\id_{\{\tau<t\}}K_t(\xi, \tau) \quad t \geq 0,
\]
where $J$ is an $\ff$-predictable bounded process and $K: \mathbb
R_{+} \times \Omega \times\mathbb N\times\mathbb R_{+} \rightarrow
\mathbb R $ is $\mathcal P \otimes  {\mathcal{N}}  \otimes
\mathcal B(\mathbb R_{+})$-measurable. Note that, since $\{t\leq
\tau\}\subset \{Z_{t-}>0\}$, $J$ can be chosen to satisfy
$J_t=J_t\id_{\{Z_{t-}>0\}}$. Since $\tau$ is a thin time,
we can rewrite the process $H$ as
\begin{align*}
H_t&=J_t\id_{\{t\leq \tau\}}+\sum_{n,k}  \id_{\{\tau<t\}} K_t(\xi,
T_n)\id_{C_{n,k}}
=J_t\id_{\{t\leq \tau\}}+\sum_{n,k}
\id_{\{T_n<t\}}\id_{\{T_n=\tau\}} \id_{\{\xi=k\}} K_t(k, T_n).
\end{align*}
Note that each process $K   (k,n) $ is $\ff$-predictable and
bounded and, since $C_{n,k}\subset \{z^{n,k}_{t-}>0\}$ by Lemma \ref{key2}, $K(k,n)$
can be chosen to satisfy $K_t(k,n)
=K_t(k,n)\id_{\{z^{n,k}_{t-}>0\}}$.
 
Let now $X$ be an $H^1(\ff)$-martingale. Then stochastic integrals $J\is X$ and $K(k,n)\is X$ are well defined and each of them is $H^1(\ff)$-martingale.
For each $k,n$ and for each bounded $\ff$-martingale $N$, by integration by parts, we have that
\begin{equation}
\label{duality}
\E\left[\id_{C^k_n}N_\infty\right]=\E\left[[z^{n,k}, N]_\infty\right]=\E\left[\cro{z^{n,k}, N}_\infty\right].
\end{equation}
Since the map $N\to \E[\id_{C^k_n}N_\infty]$ is a linear form, the duality $(H^1,BMO)$ implies that \eqref{duality} holds for any $H^1(\ff)$-martingale $N$. 
Similarly, by \cite[Proposition 1.32]{aj}, for any $H^1(\ff)$-martingale $N$, the process $\cro{N,m}$ exists and we have 
\begin{align*}
\E\left [ N_\tau\right ]=\E\left [ [N, m]_\infty\right ]
=\E\left [\cro{N, m}_\infty\right ]
\end{align*}
where $m$ is given in \eqref{martm}. 
Therefore
\begin{align*}
\E\left [\int_0^\infty H_sd X_s\right ]=&
\E\left [\int_0^\tau J_s d X_s\right ]
+\sum_{n,k}\E\left [\id_{C_n^k}\int_{0}^\infty K(k,n)_s dX_s\right]\\
=&\E\left [\int_0^\infty J_s d \cro{m,X}_s\right ]
+\sum_{n,k}\E\left [\int_{0}^\infty K(k,n)_s d\cro{z^{n,k},X}_s\right]
\end{align*}
Then, since for any predictable finite variation process $V$, $\E[\int_0^\infty h_sdV_s]=\E[\int_0^\infty \;^p h_sdV_s]$, we deduce
\begin{align*}
\E&\left [\int_0^\infty H_sd X_s\right ]=
\E\left [\int_0^\infty \frac{Z_{s-}}{Z_{s-}}\id_{\{Z_{s-}>0\}}J_s d \cro{m,X}_s\right ]\\
&\phantom{}+\sum_{n,k}\E\left [\int_{0}^\infty \frac{z^{n,k}_{s-}}{z^{n,k}_{s-}}\id_{\{z^{n,k}_{s-}>0\}}K(k,n)_s d\cro{z^{n,k},X}_s\right]\\
=&\E\left [\int_0^\tau \frac{1}{Z_{s-}}J_s d \cro{m,X}_s\right ]
+\sum_{n,k}\E\left [\id_{C_n^k}\int_{0}^\infty \frac{1}{z^{n,k}_{s-}}K(k,n)_s d\cro{z^{n,k},X}_s\right].
\end{align*}
For any $H^1(\ff)$-martingale $Y$ and $\ggz$-predictable process $G\equiv 1$ the assertion of the theorem follows as for any $s\leq t$ and $F\in \G_s$ the process $H=\id_{(s,t]}\id_{F}$ is clearly $\ggz$-predictable.
To end the proof we recall that any local martingale is locally in $H^1$.
\end{proof}

\section{Entropy of a pair $(\xi, \tau)$}
\subsection{Main result}
Theorem \ref{entropy2}, where we answer the question of Meyer about measurement of the perturbation introduced progressively in time, is  the main result of this section and constitute a generalization of \cite[Theorem 2]{entropieyor}. 
To formulate it in the context of the enlargement from Section \ref{s:2}, we need to define a more general object than \eqref{entropytau}, namely  the $\gamma$-entropy of a pair $(\xi,\tau)$  by
\begin{equation}
\label{entropytau2}
H_\gamma(\xi, \tau):=\sum_{n,k}\E\left[\id_{C_n^k}\Big[\log \frac{1}{z^{n,k}_{T_n}}\Big ]^\gamma\right]\quad \quad\gamma>0 .
\end{equation}

\begin{remark}
(a) If $\tau$ is an $\ff$-stopping time then $H_\gamma(c,\tau)=0$.
(b) If for any $n$ and $k$ the set $C^k_n$ is in $\F_{T_n}$, then $\id_{C^k_n}\log
z^{n,k}_{T_n}=\id_{C^k_n}\log \id_{C^k_n}=0$ and we do not
gain any additional information.
(c) As noted in \cite[Proposition 1.2]{thin}, the exhausting sequence $(T_n)_n$ of a thin random time is not unique. However $H_\gamma(\xi,\tau)$ is invariant under different exhausting sequences of $\tau$.
Let $T$, $T_1$ and $T_2$ be 
$\ff$-stopping times  such that $\{\tau=T\}=\{\tau=T_1\}\cup\{\tau=T_2\}$ and $\graph{T_1}\cap\graph{T_2}=\emptyset$. Then we have
\begin{align*}
\id_{\{\tau=T\}}\log^\gamma \P(\tau=T|\F_T)&=\id_{\{\tau=T_1\}}\log^\gamma \P(\tau=T|\F_T)+\id_{\{\tau=T_2\}}\log^\gamma \P(\tau=T|\F_T)\\
&=\id_{\{\tau=T_1\}}\log^\gamma \P(\tau=T_1|\F_{T_1})+\id_{\{\tau=T_2\}}\log^\gamma \P(\tau=T_2|\F_{T_2}).
\end{align*}
\end{remark}

The $\gamma$-entropy of $(\xi,\tau)$ reveals to be an adequate notion to treat the stability of martingale spaces with respect to progressive enlargement of filtration with partition. 
In this section we work under standing assumption that 
\begin{equation}
\label{C}
\tag{{\bf C}} \quad \textrm{all $\ff$-martingales are continuous}
\end{equation}
For any $p\in [1,\infty)$, we denote $H^p$ and $S^p$ the
Banach spaces consisting respectively of
continuous local martingales and continuous semimartingales equipped with the following norms:\\
(a) a continuous $\ff$-local martingale $X$ belongs to $H^p$ if
$$\norm{X}_{H^p}:=\norm{\cro{X}^{1/2}_\infty}_{L^p}<\infty;$$
(b) a continuous $\ff$-semimartingale $X$, with canonical decomposition $X=M+V$, belongs to $S^p$ if
$$\norm{X}_{S^p}:=\norm{\cro{M}^{1/2}_\infty}_{L^p}+\norm{\int_0^\infty|dV_t|}_{L^p}<\infty.$$

We are ready to state the main result of this section. 

\begin{theorem}
\label{entropy2}
Let $(C_n^k)_{n,k}$ be an $\F_\infty$-measurable partition, $\tau$ be a thin random time with exhausting sequence $(T_n)_n$ and $\ggz$ be given by \eqref{xitau}.
Assume \eqref{C} and let $r\in[1,\infty)$, $p,\gamma>0$ satisfy $\frac{1}{r}=\frac{1}{p}+\frac{1}{\gamma}.$ The the following conditions are equivalent:
\begin{itemize}
\item[(a)] for each $\ff$-local martingale $Y$ and each $\ggz$-predictable process $G$, the $\ggz$-semimartingale $X:=\int G dY$ satisfies:
\begin{equation*}
\norm{X}_{S^r(\ggz)}\leq C_{p,r} \norm{\cro{X}^{1/2}_\infty}_{L^p};
\end{equation*}
\item[(b)]
$H_{\gamma/2}(\xi, \tau)<\infty$. 
\end{itemize}
In particular, if the conditions (a) and (b) are satisfied, then $H^p(\ff)$ is continuously embedded in $S^r(\ggz)$.
\end{theorem}

\subsection{Proof of Theorem \ref{entropy2}}
By Theorem \ref{decomposition2}, under assumption \eqref{C}, each $X$ of the form $X=\int GdY$ where $Y$ is an $\ff$-local martingale and $G$ is an $\ggz$-predictable process, has the decomposition
\begin{equation}
\label{YaYb}
X=\wh X+\cro{X,Y^b}+\cro{X,Y^a}
\end{equation}
where $Y^b$ and $Y^a$ are $\ggz$-local martingales given by 
\begin{align}
\label{eq:YY}
Y^b_t:=\int_0^{\tau\land t}\frac{1}{Z_{s-}}d\wh m_s \quad\textrm{and} \quad Y^a_t:=\sum_{n,k} \id_{C^k_n}\int_0^t\id_{\{s>T_n\}}\frac{1}{z^{n,k}_s}d\,\wh z^{\;n,k}_s
\end{align}
with $\wh X$, $\wh m$ and $\wh z^{\,n,k}$ the $\ggz$-local martingale parts from the Doob--Meyer decomposition of corresponding $\ggz$-semimartingales $X$, $m$ and $z^{n,k}$. 
Since $r<p$ and $X$ is continuous, it always holds that 
$\norm{\cro{\wh X}^{1/2}_\infty}_{L^r}=\norm{\cro{X}^{1/2}_\infty}_{L^r}\leq C\norm{\cro{X}^{1/2}_\infty}_{L^p}$.
Thus, (a) is equivalent to showing that 
\begin{equation}
\label{e1}
\norm{\int_0^\infty|d\cro{X,Y^b}_t+\cro{X,Y^a}_t|}_{L^r}\leq C_{p,r} \norm{\cro{X}^{1/2}_\infty}_{L^p}
\end{equation}
holds for any $X$. 
By \cite[Lemma 2]{entropieyor}, it is further equivalent to 
$\E\!\left[\langle Y\rangle_\infty^{\gamma/2}\right] < \infty.$
Using the fact that stochastic intervals $\Rbrack 0,\tau\Lbrack$ 
and $\Lbrack T_n,\infty\Rbrack\,\cap\, C^k_n$ 
for $n\in\mathbb N$ are pairwise disjoint, 
inequality \eqref{e1} holds for any adequate $X$ 
if and only if 
\begin{align*}
\label{e2}
\E\left [\cro{Y^b}_\infty^{\gamma/2}\right]
&<\infty \quad \textrm{and}\quad
\E\left [\cro{Y^a}_\infty^{\gamma/2}\right]
<\infty.
\end{align*}
Firstly we show that 
\begin{equation}
\label{yb}
\norm{\cro{Y^b}^{1/2}_\infty}_{L^\gamma}<\infty\quad \forall\;\gamma>0.
\end{equation}
By \cite[Remark 5.1 2) p.123]{yorin}, since, for $\gamma>2$, $x^{\gamma/2}$ is moderate Orlicz function, we have
\[
\Big|\!\Big|\left[\int_0^\tau\frac{1}{Z^2_s}d\cro{m}_s\right]^{1/2}\Big|\!\Big|_{L^\gamma}
\leq C\left[1+\Big|\!\Big|\left[\log\frac{1}{U}\right]^{1/2}\Big|\!\Big|_{L^\gamma}\right]
\] 
where $U$ is random variable with uniform distribution on $[0,1]$.
Note that $\int_0^1(-\log x)^{\gamma/2}dx<\infty$ and thus, by the fact that $L^\mu\subset L^\gamma$ for $0<\mu<\gamma<\infty$, the \eqref{yb} holds.

Showing that $\E\left [\cro{Y^a}_\infty^{\gamma/2}\right]<\infty$ if and only if $H_{\gamma/2}(\xi,\tau)<\infty$ for $\gamma=2$ is a simpler special case, which will be useful afterwards in Section \ref{s:utility}, and we start with it.
By properties of dual predictable projection, we have that
\begin{align*}
\label{eq:Ya}
\E\left [\cro{Y^a}_\infty\right]&=\sum_{n,k} \E\left[\int_{T_n}^\infty \id_{C^k_n}\frac{1}{(z^{n,k}_t)^2} d\cro{z^{n,k}}_t\right]
=\sum_{n,k} \E\left[\int_{T_n}^\infty \frac{1}{z^{n,k}_t} d\cro{z^{n,k}}_t\right].
\end{align*}
Consider the function $f: \R_+\to\R_+$ defined as $f(x)=x-x\log x$ for $x>0$ and $f(0)=0$.
Then, It\^o's formula for $z^n$ implies
\[
\id_{C^k_n}=z^{n,k}_{t}-z^{n,k}_{t}\log z^{n,k}_{t}-\int_{t}^\infty \log z^{n,k}_s d z^{n,k}_s-\frac{1}{2}\int_{t}^\infty \frac{1}{z^{n,k}_s} d \cro{z^{n,k}}_s.
\]
We deduce, by taking conditional expectation with respect to $\F_{t}$, that
\begin{equation}
\label{xlogx}
\E\left[\int_{t}^\infty \frac{1}{z^{n,k}_s} d \cro{z^{n,k}}_s|\F_{t}\right]
=2z^{n,k}_{t}\log \frac{1}{z^{n,k}_{t}}.
\end{equation}
Finally, 
\begin{align}
\label{Ya}
\E\left[\cro{Y^a}_\infty\right]
&=\sum_{n,k}\E\left[\int_{T_n}^\infty \frac{1}{z^{n,k}_t} d\cro{z^{n,k}}_t\right]
=-2\sum_{n,k}\E\left[z^{n,k}_{T_n}\log z^{n,k}_{T_n}\right]
=2H(\xi,\tau)<\infty.
\end{align}

In order to complete the proof it remains to show that $H_\gamma(\xi,\tau)<\infty$ if and only if 
$\E\left [\cro{Y^a}^{\gamma}\right]<\infty$ for $\gamma>0$. 
First, Lemma \ref{lem1} implies that it is equivalent to proving  
$$H_\gamma(\xi,\tau)<\infty\quad  \textrm{if and only if} \quad
\E\left [\left [\log \frac{1}{I}\right]^{\gamma}\right]
<\infty$$
where $I$ is defined in \eqref{I}.
To show the latter equivalence, note that, by Lemma \ref{lem2}, we have
\begin{align*}
\E\left [\left [\log \frac{1}{I}\right]^{\gamma}\right]
&=\sum_{n,k} \E\left [\E\left [\id_{C^k_n}\left [\log \frac{1}{I^k_n}\right]^{\gamma}\Big |\F_{T_n}\right]\right ]\\
&=\sum_{n,k} \E\left [\id_{C^k_n}
\frac{1-z^{n,k}_{T_n}}{z^{n,k}_{T_n}}\int_0^{z^{n,k}_{T_n}}\left[\log\frac{1}{\beta}\right]^\gamma\frac{1}{(1-\beta)^2}d\beta
\right ]. 
\end{align*}
Denoting by $z=\sum_{n,k}\id_{C^k_n}z^{n,k}_{T_n}$, taking any $\varepsilon\in(0,1)$ and defining $f(x)=\int_0^{x}\left[\log\frac{1}{\beta}\right]^\gamma d\beta$ for $x\in(0,1)$, we further obtain
\begin{align*}
\E&\left [\left [\log \frac{1}{I}\right]^{\gamma}\right]
=\E\left [
\frac{1-z}{z}\int_0^{z}\left[\log\frac{1}{\beta}\right]^\gamma\frac{1}{(1-\beta)^2}d\beta
\right ]\\
&\leq 
\E\left [\id_{\{z>\varepsilon\}}
\left[\frac{f(\varepsilon)}{\varepsilon(1-\varepsilon)^2}
+\frac{1}{1-\varepsilon}\left[\log\frac{1}{\varepsilon}\right]^\gamma\right]
\right ]
+\E\left [\id_{\{z\leq\varepsilon\}}
\frac{f(z)}{z(1-\varepsilon)^2}\right]\\
&\leq C_1+C_2\E\left [\id_{\{z\leq\varepsilon\}}
\left[\log\frac{1}{z}\right]^\gamma \right]\\
&\leq C_1+C_2H_\gamma(\xi,\tau). 
\end{align*}
Thus we conclude that $
H_\gamma(\xi,\tau)<\infty$ if and only if $\E\left [\left [\log \frac{1}{I}\right]^{\gamma}\right]<\infty$
and the proof is complete.

\section{Additional logarithmic utility of a pair $(\xi,\tau)$}
\label{s:utility}

Proposed measurement, given in \eqref{entropytau2}, of the perturbation to the random system of a partition released progressively in time, is further justified by the connection to the following logarithmic utility maximization problem. We suppose that condition \eqref{C} holds in this section.

Let us consider a continuous financial market model with two assets: a risk-free asset $S^0\equiv 0$ and a risky asset $S$ satisfying:
\begin{equation}
\label{price_s}
dS_t=S_t\left(dM_t+\lambda d\langle M\rangle_t\right)\quad S_0>0    
\end{equation}
where $M$ be a continuous $\mathbb F$-local martingale, and $\lambda$ is an $\mathbb F$-predictable process square integrable w.r.t. $M$.
Note that \eqref{price_s} is so called structure condition which is closely related to the market viability, as well as solution to the log-utility problem (see e.g. \cite{KK, hulley}).

By Theorem \ref{decomposition2}, \eqref{price_s} can be written as
\begin{equation*}
\label{price_sg}
dS_t=S_t\left(d\widehat M_t + \id_{\{t\leq\tau\}}\frac{1}{Z_{t-}}d\cro{M,m}_s +\sum_{n,k} \id_{C_{n,k}}
\id_{\{t>T_n\}}\frac{1}{z^{n,k}_{t-}}d\cro{M,z^{n,k}}_t+\lambda d\langle M\rangle_t\right), 
\end{equation*}
which, by introducing 
\begin{align}
\label{eq:L}
m_t&=\int_0^t\varphi^m_sdM_s+L^m_t\quad\textrm{and}\quad
z^{n,k}_t=\int_0^t\varphi^{n,k}_sdM_s+L^{n,k}_t,
\end{align}
where $L^m$ and $L^{n,k}$ are $\ff$-martingales orthogonal to $M$, and denoting
\begin{align*}
\mu&=\id_{\Rbrack 0, \tau\Lbrack}\frac{\varphi^m}{Z_{s-}}+\sum_{n,k}\id_{C_n^k}\id_{\Lbrack T_n,T\Lbrack}\frac{\varphi^{n,k}}{z^{n,k}},
\end{align*}
becomes
\begin{equation}
\label{price_sgn}
dS_t=S_t\left(d\widehat M_t +(\lambda+\mu) d\langle \widehat M\rangle_t\right),
\end{equation}
where $\widehat M$ is a $\ggz$-local martingale.

Let us fix the time horizon $T\leq \infty$ and the initial wealth $x>0$.
An $\ff$-portfolio (resp. $\ggz$-portfolio) process is an 
$\ff$-predictable ($\ggz$-predictable) process
$\pi=(\pi_s)_{s\in[0,T]}$
such that
$
\int_0^T \pi_s^2\, d\langle M\rangle_s < \infty
$
a.s.
Wealth process $V(x,\pi)$  associate to a portfolio process $\pi$, 
is given by
\begin{equation}
\label{wealth}
dV_t(x,\pi)=\pi_t\, V_t(x,\pi)\,
\frac{dX_t}{X_t},
\quad t\in[0,T] \qquad \textrm{and}\quad  V_0(x,\pi)=x.
\end{equation}
Process $(\pi_t)$ describes the proportion of total wealth at time $t$
invested in the risky asset and \eqref{wealth} is the self-financing condition (see e.g. \cite{karatzas2021portfolio}).

The regular agent maximizes the following log-utility problem:
\begin{equation}
\sup_{\pi \in \mathcal A^{\ff}(x,T)}\E[\log(V_T(x,\pi))],
\end{equation}
while the insider the following:
\begin{equation}
\sup_{\pi \in \mathcal A^{\ggz}(x,T)}\E[\log(V_T(x, \pi))],
\end{equation}
where $\mathcal A^{\mathbb H}(x,t)$, for $\mathbb H\in \{\ff,\ggz\}$, denotes the class of admissible $\mathbb H$-portfolio processes and is given as $$
\mathcal{A}^{\mathbb H}(x,t)
:=
\left\{
\pi \mid
\pi \text{ is an }\mathbb H\text{-portfolio process and }
E[\log V_t(x,\pi)]<\infty
\right\}
.$$
The connection to entropy of the pair $(\tau,\xi)$ is given in the following result. 
It can be interpreted as additional logarithmic utility of the insider which is realized \emph{after} time $\tau$.

\begin{theorem}
\label{th:logu}
Assume that $\tau\leq T$, and that $L^m$ and $L^{n,k}$ given in \eqref{eq:L} are constant. Then, the additional logarithmic utility of the insider, i.e.
\begin{equation*}
aU(\ff,\ggz):=\sup_{\pi \in \mathcal A^{\ggz}(x,T)}\E[\log(V_T(x,\pi))]-\sup_{\pi \in \mathcal A^{\ff}(x,T)}\E[\log(V_T(x,\pi))]
\end{equation*}
is given by 
\begin{equation*}
aU(\ff,\ggz)=
\E\left[\langle Y^b_\tau\rangle\right]+
H(\xi,\tau)
\end{equation*}
where $H(\xi,\tau)$ is the entropy of $(\xi,\tau)$ and $Y^b$ defined in \eqref{eq:YY}.
\end{theorem}
\begin{proof}
By Theorem 3.5.1. in \cite{AIS}, optimal strategies are given by
\begin{equation}
\pi^{\ff}_t := \lambda_t, \qquad \pi^{\ggz}_t := \lambda_t+\mu_t,
\end{equation}
for the regular agent and insider, respectively,
and the additional logarithmic utility, $aU(\ff,\ggz)$, equals
\begin{equation}
aU(\ff,\ggz)=\E\!\left[
\int_0^T
(\lambda_t+\mu_t)^2\, d\langle M\rangle_t \, 
\right]-
\frac12
\E\!\left[
\int_0^T
\lambda^2_t\, d\langle M\rangle_t \, 
\right].
\end{equation}
Since $dM=d\widehat M+\mu d\langle M\rangle $, 
$$
\E\!\left[
\int_0^T
\lambda_t\mu_t\, d\langle M\rangle_t \, 
\right]=
\E\!\left[
\int_0^T
\lambda_t\, d M_t \, 
\right]-\E\!\left[
\int_0^T
\lambda_t\, dM_t \, 
\right]=0,
$$
where last equality follows from the martingale property. So, $aU(\ff,\gg)$ simplifies to 
\begin{align*}
aU(\ff,\ggz)&=
\frac{1}{2}\E\!\left[\int_0^T\mu^2_t\, d\langle M\rangle_t \, \right]\\
&=\frac12\E\left[\int_0^\tau\frac{(\varphi^m)^2_t}{Z^2_{t-}}d\langle M\rangle\right]
+\frac{1}{2}\sum_{n,k}\E\left[\id_{C^k_n}\int_{T_n}^T\frac{(\varphi^{n,k})^2_t}{(z^{n,k}_{t-})^2}d\langle M\rangle\right]\\
&=\frac12\E\left[\langle Y^b_\tau\rangle\right]
+\frac12\E\left[\langle Y^a_T\rangle\right]\\
&=\E\left[\langle Y^b_\tau\rangle\right]+
H(\xi,\tau)
\end{align*}
with $\ggz$-martingales $Y^b$ and $Y^a$ from \eqref{eq:YY}, and last equality holding by \eqref{Ya}.
\end{proof}

\appendix
\section{Optional projections}
\label{projections}

First we recall the definition of optional projection, see \cite[Theorems 5.1 and 5.2]{chinois} and \cite[p.264-265]{3M}. For definition of dual optional projection see \cite[p.265]{3M}, \cite[Chapter 3 Section 5]{protter}, \cite[Chapter 6 Paragraph 73 p.148]{dell58}, \cite[Sections 5.18, 5.19]{chinois}. 
We point out that the convention we use here allows a jump at $0$.

\begin{definition}
Let $X$ be a measurable bounded (or positive) process.
The optional projection of $X$ is the unique optional process $\ooo{X}$ such that for every stopping time $T$ we have
$$\E\left[X_T\id_{\{T<\infty\}}|\F_T\right]=\ooo{X}_T\id_{\{T<\infty\}} \quad \textrm{a.s.}.$$
\end{definition}

\begin{definition}
\label{dual}
Let $V$ be a c\`adl\`ag pre-locally integrable variation process (not necessary adapted).
The dual optional projection of $V$ is the unique optional process $\oo{V}$ such that for every optional process $H$ we have
$$\E\left[\int_{[0,\infty)} H_sdV_s\right]=\E\left[\int_{[0,\infty)} H_sd\oo{V}_s\right].$$
\end{definition}
\section{Auxiliary lemmas}
Here we present generalisations of two auxiliary lemmas, Lemma 3 and Lemma 4, from \cite{entropieyor}. 
They serve to prove Theorem \ref{entropy2}.
\begin{lemma}
\label{lem1}
For all $\gamma>0$ there exist $c_\gamma$ and $C_\gamma$ such that 
\begin{equation}
\label{ine:inf}
c_\gamma\E\left [\cro{\ya}^\gamma_\infty\right ]
\leq \E\left [\left [\log \frac{1}{I}\right ]^\gamma\right ]
\leq C_\gamma\E\left [\cro{\ya}^\gamma_\infty\right ]
\end{equation}
where $I$ is defined as
\begin{equation}
\label{I}
I:=\sum_{n,k} \id_{C^k_n} I^k_n\quad \textrm{where}\quad I^k_n:=\inf_{t\geq T_n} z^{n,k}_t
\end{equation}
\end{lemma}

\begin{proof}
\emph{Step 1.} We first prove the first inequality in \eqref{ine:inf}. 
We have that
\begin{align*}
\id_{\{t\geq T_n\}\cap C^k_n}& \E\Big [\cro{\ya}_\infty-\cro{\ya}_t|\Gz_{t}\Big ]
=\id_{\{t\geq T_n\}\cap C^k_n} \E\Big [\int_t^\infty \frac{1}{(z^{n,k}_s)^2}d\cro{z^{n,k}}_s|\Gz_{t}\Big ]\\
&=\id_{\{t\geq T_n\}\cap C^k_n}\frac{1}{z^{n,k}_t} \; \E\Big [\id_{C^k_n}\int_t^\infty \frac{1}{(z^{n,k}_s)^2}d\cro{z^{n,k}}_s|\F_{t}\Big ]\\
&=\id_{\{t\geq T_n\}\cap C^k_n}\frac{1}{z^{n,k}_t}\; \E\Big [\int_t^\infty \frac{1}{z^{n,k}_s}d\cro{z^{n,k}}_s|\F_{t}\Big ]\\
&=2\id_{\{t\geq T_n\}\cap C^k_n} \log\frac{1}{z^{n,k}_t}
\leq 
2\id_{C^k_n}\log\frac{1}{I^k_n}
\end{align*}
where the second equality follows from Lemma \ref{key2}, 
the third one from dual predictable projection properties 
and the four one from \eqref{xlogx}.
Therefore, for each $\mu\in(0, 1]$ we deduce that
\begin{align}
\nonumber
\id_{\{t\geq T_n\}\cap C^k_n} \E\Big [\Big[\cro{\ya}_\infty-\cro{\ya}_t\Big ]^\mu|\Gz_{t}\Big ]
&\leq\id_{\{t\geq T_n\}\cap C^k_n} \Big [\E\Big[\cro{\ya}_\infty-\cro{\ya}_t |\Gz_{t}\Big ]^\mu\\
\label{eq1}
&\leq 2^\mu \id_{C^k_n} \Big [\log\frac{1}{I^k_n}\Big ]^\mu.
\end{align}
Consequently, the required inequality for $\gamma\in(0,1]$ follows by
\begin{align*}
\E(\cro{\ya}^\gamma_\infty)
&=\sum_{n,k}\E\left[\id_{C^k_n} \E\Big [\Big[\cro{\ya}_\infty-\cro{\ya}_{T_n}\Big ]^\gamma|\Gz_{T_n}\Big ]\right]\\
&\leq\sum_{n,k}\E\left[\sup_t\id_{\{t\geq T_n\}\cap C_n} \E\Big [\Big[\cro{\ya}_\infty-\cro{\ya}_t\Big ]^\gamma|\Gz_{t}\Big ]\right]\\
&\leq 2^\gamma \sum_{n,k}\E\left[\id_{C^k_n} \Big [\log\frac{1}{I^k_n}\Big ]^\gamma\right]
=2^\gamma \E\left[\left[\log \frac{1}{I}\right]^\gamma\right]
\end{align*}
and for $\gamma>1$ follows by
\begin{align*}
\E(\cro{\ya}^\gamma_\infty)
&\leq \gamma^\gamma\sum_{n,k}\E\left [\sup_t\id_{\{t\geq T_n\}\cap C^k_n} \Big [\E\Big [\cro{\ya}_\infty-\cro{\ya}_t|\Gz_{t}\Big ]\Big ]^\gamma\right]\\
&\leq (2\gamma)^\gamma \sum_{n,k}\E\left[\id_{C^k_n} \Big [\log\frac{1}{I^k_n}\Big ]^\gamma\right]
=(2\gamma)^\gamma \; \E\left[\left[\log \frac{1}{I}\right]^\gamma\right],
\end{align*}
where the first inequality is due to Burkholder-Gundy inequality for terminal value of increasing process and supremum of the associated potential (\cite{dell58} p.188) and the second inequality is due to \eqref{eq1} for $\mu=1$.

\emph{Step 2.} We now prove the second inequality in \eqref{ine:inf}. 
Let $\mu\in(0,1]$ and $p> 1$ and consider $\ggz$-martingale $M_t:=\E(\cro{\ya}^\mu_\infty|\Gz_t)$. 
Firstly we will show that
\begin{align}
\label{i2}
\left [\log \frac{1}{I}\right ]^\mu
\leq C_\mu\Big[1+\sup_t M_t\Big ].
\end{align}
By It\^o's lemma and decomposition \eqref{thinH2} applied to $z^{n,k}$ we obtain that on $\{t\geq T_n\}\cap C^k_n$
\begin{align*}
\log\frac{1}{z^{n,k}_t}&
=\int_t^\infty \frac{1}{z^{n,k}_s} d \wh z^{\;n,k}_s+\frac{1}{2}\int_t^\infty \frac{1}{(z^{n,k}_s)^2} d \cro{z^{n,k}}_s.
\end{align*}
Next, by taking conditional expectation w.r.t $\Gz_t$ and using inequality $|x+y|^\mu\leq |x|^\mu+|y|^\mu$ for $\mu\in(0,1]$, we deduce that on $\{t\geq T_n\}\cap C^k_n$
\begin{align*}
&\left[\log\frac{1}{z^{n,k}_t}\right]^\mu 
\leq \E\left [\left |\int_t^\infty \frac{1}{z^{n,k}_s} d \wh z^{\;n,k}_s\right |^\mu \big |\Gz_t\right]
+\frac{1}{2^\mu}\E\left [\left [\int_t^\infty \frac{1}{(z^{n,k}_s)^2} d \cro{z^{n,k}}_s\right]^\mu\big |\Gz_t\right]\\
 &
\leq C\left [\E\left [\left [\int_t^\infty \frac{1}{(z^{n,k}_s)^2} d \cro{z^{n,k}}_s\right ]^{\mu/2} \big |\Gz_t\right]
+\E\left [\left [\int_t^\infty \frac{1}{(z^{n,k}_s)^2} d \cro{z^{n,k}}_s\right]^\mu\big |\Gz_t\right]\right]
\\
 &
\leq C\left[\E\left [\cro{\ya}_\infty^{\mu/2} \big |\Gz_t\right]
+\E\left [\cro{\ya}_\infty^\mu\big |\Gz_t\right]\right]
\end{align*}
where in the second inequality we have used Burkholder-Davis-Gundy inequality for continuous local martingales (see \cite[IV-42, p.93]{williamsvol2}) applied to $\wh \ff$-local martingale $N\id_{\wh F}$ for any $\wh F\in\wh \F_0$ where
\begin{align*}
N_T&:=\int_t^{T+t} \frac{1}{z^{n,k}_s} d \wh z^{\;n,k}_s\quad\textrm{and}\quad\wh\gg:=(\wh \Gz_T)_{T\geq 0}\quad\textrm{with}\quad \wh\Gz_T:=\Gz_{T+t}.
\end{align*}
Finally, using inequality $x+x^2\leq 2 x^2+1/4$ to $x=\cro{\ya}_\infty^{\mu/2}$ we conclude that on $\{t\geq T_n\}\cap C^k_n$
\begin{align*}
\left[\log\frac{1}{z^{n,k}_t}\right]^\mu
&
\leq C\left[1+\E\left [\cro{Y^a}_\infty^\mu\big |\Gz_t\right]\right ]
\end{align*}
and by taking supremum over $t$ and summing over $n$ the inequality \eqref{i2} follows.
In order to prove general case, relying on \eqref{i2} and the inequality $(1+x)^p\leq 2^{p-1}(1+x^p)$, we obtain
\begin{align*}
\left [\log \frac{1}{I}\right ]^{\mu p}
\leq C_\mu^p\, 2^{p-1}\Big[1+\sup_t M^p_t\Big ].
\end{align*}
Then, by taking expectations and applying Doob's maximal inequality to $M$ we obtain
\begin{align*}
\E\left [\Big[\log \frac{1}{I}\Big ]^{\mu p}\right ]
&\leq C_\mu^p\, 2^{p-1}\Big[1+\E(\sup_t M^p_t)\Big ]
\leq C_\mu^p\, 2^{p-1}\Big[1+\frac{p^p}{(p-1)^p}\E( \cro{\ya}^{\mu p}_\infty)\Big ].
\end{align*}
That completes the proof since any $\gamma=\mu p>0$ can be obtained with $\mu\in(0,1]$ and $p>1$.
\end{proof}

\begin{lemma}
\label{lem2}
Let, for each $n$ and $k$, $\Q^{n,k}$ be an absolutely continuous measure given by $\frac{d \Q^{n,k}}{d \P}=\frac{\id_{C^k_n}}{z^{n,k}_0}$.
Then, for each $\F_{T_n}$-measurable random variable $\beta$ with values in random interval $(0, z^{n,k}_{T_n})$, we have 
$$\Q^{n,k}(I^k_n<\beta|\F_{T_n})=\frac{\beta}{1-\beta}\frac{1-z^{n,k}_{T_n}}{z^{n,k}_{T_n}}$$
where $I^k_n$ is defined in \eqref{I}. 
Or, equivalently 
$$\Q^{n,k}(I^k_n\in d\beta |\F_{T_n})=\frac{1-z^{n,k}_{T_n}}{z^{n,k}_{T_n}}\frac{1}{(1-\beta)^2}\id_{\{0<\beta<z^{n,k}_{T_n}\}}.$$
\end{lemma}

\begin{proof}
For any $\F_{T_n}$-measurable random variable $\beta$ with values in random interval $(0, z^{n,k}_{T_n})$ we define an $\ff$-stopping time $\sigma_n^\beta$ by 
$\sigma_n^\beta:=\inf\{t\geq T_n: z^{n,k}_t<\beta\}.$
Then we compute
\begin{align*}
\P(\{I^k_n<\beta\}\cap C^k_n|\F_{T_n})&=
\P(\{\sigma_n^\beta< \infty\}\cap C^k_n|\F_{T_n})\\
&=
\E_\P(\P(\{\sigma_n^\beta< \infty\}\cap C^k_n|\F_{\sigma_n^\beta})|\F_{T_n})\\
&=
\E_\P(\id_{\{\sigma_n^\beta< \infty\}}z^{n,k}_{\sigma_n^\beta}|\F_{T_n})
=
\beta \,\P({\{\sigma_n^\beta< \infty\}}|F_{T_n})\\
&=
\beta \,\P({\{I^k_n<\beta\}}\cap C^k_n|F_{T_n})
+\beta \,\P({\{I^k_n<\beta\}}\cap (\Omega \,\backslash C^k_n)|F_{T_n}).
\end{align*}
Since $\{I^k_n<\beta\}\cap (C^k_n)^c=(C^k_n)^c$ we deduce that
$$(1-\beta)\P(\{I^k_n<\beta\}\cap C^k_n|\F_{T_n})=\beta(1-z^{n,k}_{T_n})$$
and therefore, by Bayes' rule, we have 
$$\Q^{n,k}(I^k_n<\beta|\F_{T_n})=\frac{\P(\{I^k_n<b\}\cap C^k_n|\F_{T_n})}{z^{n,k}_{T_n}}=\frac{\beta}{1-\beta}\frac{1-z^{n,k}_{T_n}}{z^{n,k}_{T_n}}$$
which completes the proof.
\end{proof}

\section*{Acknowledgments}
The author acknowledges the generous support of the Australian Research Council through grant DP220103106.


\end{document}